\documentclass{amsart}
\usepackage{xcolor}
\usepackage{graphicx,amssymb}
\usepackage[scale=0.75]{geometry}
\usepackage[all]{xy}

\usepackage[hidelinks,pagebackref]{hyperref}

\theoremstyle{plain}
\newtheorem{lem}{Lemma}[section]
\newtheorem{cor}[lem]{Corollary}

\newtheorem{prop}[lem]{Proposition}
\theoremstyle{definition}
\newtheorem{ex}[lem]{Example}
\newtheorem{rem}[lem]{Remark}
\newtheorem*{thmmain}{Theorem}
\newtheorem*{exmain}{Example}
\newtheorem*{remmain}{Remarks}

\newcommand\ZZ{\mathbb{Z}} 
\newcommand{\BF}{\bar{F}} 

\newcommand\OO{\mathcal{O}} 
\newcommand\PP{\mathbb{P}} 
\newcommand{\Aa}{\mathcal{A}} 
\newcommand{\II}{\mathcal{I}} 
\newcommand{\FF}{\mathcal{F}} 

\newcommand{\id}{\mathrm{id}} 
\newcommand{\res}{\mathrm{res}} 
\newcommand{\Ga}{\Gamma} 
\newcommand{\LL}{\mathcal{L}} 

\newcommand\MM{\operatorname{M}} 
\newcommand\Spec{\operatorname{Spec}} 
\newcommand{\Gal}{\operatorname{Gal}}
\newcommand{\PGL}{\operatorname{PGL}}
\newcommand{\SB}{\operatorname{SB}}
\newcommand{\Aut}{\operatorname{Aut}}
\newcommand{\Hom}{\operatorname{Hom}}
\newcommand{\CH}{\operatorname{CH}}
\newcommand{\End}{\operatorname{End}}
\newcommand{\Sym}{\operatorname{Sym}}
\newcommand{\Tan}{\operatorname{Tan}}
\newcommand{\Span}{\operatorname{span}}
\newcommand{\Gr}{\operatorname{Gr}}
\newcommand{\cancel}[1]{\,\not{\!#1}}

\def\pentagon{{\setlength{\unitlength}{0.5pc}
\begin{picture}(2,2)
\put(0.5,-.2){\circle*{0.5}}
\put(1.5,-.2){\circle*{0.5}}
\put(1.81,0.75){\circle*{0.5}}
\put(0.19,0.75){\circle*{0.5}}
\put(1,1.34){\circle*{0.5}}
\qbezier(0.5,-.2)(0.5,-.2)(1.5,-.2)
\qbezier(0.5,-.2)(0.5,-.2)(0.19,0.75)
\qbezier(1.5,-.2)(1.5,-.2)(1.81,0.75)
\qbezier(1,1.34)(1,1.33)(0.19,0.75)
\qbezier(1,1.34)(1,1.33)(1.81,0.75)
\end{picture}}}

\author[R.~Xiong]{Rui Xiong}
\address[Rui Xiong]{Department of Mathematics and Statistics, University of Ottawa, 150 Louis-Pasteur, Ottawa, ON, K1N 6N5, Canada}
\email{rxiong@uottawa.ca}
\urladdr{https://cubicbear.github.io}

\author[K.~Zainoulline]{Kirill Zainoulline}
\address[Kirill Zainoulline]{Department of Mathematics and Statistics, University of Ottawa, 150 Louis-Pasteur, Ottawa, ON, K1N 6N5, Canada}
\email{kirill@uottawa.ca}
\urladdr{https://kirillmath.ca}

\begin{document}

\title{Motivic Lefschetz theorem for twisted Milnor hypersurfaces}

\begin{abstract}
We show that the Grothendieck-Chow motive of a smooth hyperplane section $Y$ of an inner twisted form $X$ of a Milnor hypersurface splits as a direct sum of shifted copies of the motive of the Severi-Brauer variety of the associated cyclic algebra $A$ and the motive of its maximal commutative subfield $L\subset A$. The proof is based on the non-triviality of the (monodromy) Galois action on the equivariant Chow group of $Y_L$.  
\end{abstract}

\thanks{Both authors were partially supported by NSERC Discovery grant RGPIN-2022-03060}

\maketitle

\section{Introduction}

Milnor hypersurfaces, their twisted forms and hyperplane sections appear in many different contexts and provide an important source of examples of algebraic varieties in various areas of mathematics. In the present notes we focus on the study of their Grothendieck-Chow motives \cite{Ma68} using the (monodromy) action of the Galois group on the respective equivariant Chow groups \cite{To97,Br97}. 

\subsection{\it Twisted Milnor hypersurfaces}
Following \cite[2.5.3]{LM07} the Milnor hypersurface $X_0$ of dimension $2n-3$, $n\ge 3$ over a field $F$ is a zero divisor of the line bundle  $\OO(1)\boxtimes \OO(1)$ on the product of projective spaces $\PP^{n-1}\times\PP^{n-1}$ or, equivalently, it is the flag (incidence) variety 
\[
X_0=\{V_1\subset V_{n-1}\subset F^{n} \},\; \text{ where } \dim_F V_i=i.
\]
There is another description of $X_0$ in terms of right ideals \cite[\S1B]{KMRT}:
\[
X_0=\{I_1\subset I_{n-1}\subset \End(F^n)\},\;\text{ where } \dim_F I_i=ni.
\]
The identification is given by  $(V_1\subset V_{n-1})\mapsto (I_1\subset I_{n-1})$ with $I_i=\Hom(F^n,V_i)\subset \End(F^n)$.

Using the Bruhat decomposition one can show that the Grothendieck-Chow motive of $X_0$ splits as a direct sum of (shifted) Tate motives \cite{Ko91}:
\[
\MM(X_0)=\bigoplus_{i=0}^{n-2}\big(\ZZ(i)\oplus \ZZ(2n-3-i)\big)^{\oplus (i+1)},\;
\text{ where }\ZZ(i)=\ZZ(1)^{\otimes i}\text{ and } \ZZ(0)=\MM(\Spec F).
\]

Let $X$ be a projective homogeneous $G$-variety over $F$, where $G$ is a linear algebraic group over $F$ such that $G_{\BF}\simeq \PGL_n$ and $X_{\BF}\simeq (X_0)_{\BF}$, where $X_{\BF}=X\!\times_{\Spec F}\Spec \BF$ denotes the base change to the separable closure $\BF$ of $F$. We call $X$ a twisted Milnor hypersurface.

Such $X$ provides an example of a $G$-flag variety in the sense of \cite[\S1]{MPW1} or the so-called twisted flag variety.  The Galois group $\Ga=\Gal(\BF/F)$ acts naturally on $G(\BF)$, its maximal torus  and the respective Weyl group $W$ via the so-called $\star$-action (see \cite{CM06,MPW1}). The kernel of this $\star$-action defines the smallest field extension $F_G/F$ such  that $G_{F_G}$ is of inner type (it has no non-split quasi-split forms). There are two basic families of twisted Milnor hypersurfaces: the inner hypersurfaces ($F_G=F$) and the quasi-split hypersurfaces ($G$ is quasi-split).

For the study of inner hypersurfaces we refer to \cite[\S2]{CPSZ}). In this case $G=\Aut_F(A)$ for some central simple algebra $A$ over $F$, and by \cite[Cor.2.3]{CPSZ} the hypersurface $X$ can be identified  with a certain Grassmann bundle over the Severi-Brauer variety $\SB(A)$. Its motive $\MM(X)$ then splits as a direct sum of shifted copies of the motive $\MM(\SB(A))$. For properties of the latter see \cite{Ka00,Ka96}. Observe that both $\MM(X)$ and $\MM(\SB(A))$ do not contain Artin motives (motives of non-trivial field extensions) as direct summands. 

In the quasi-split case, $\Gal(F_G/F)\simeq \ZZ/2\ZZ$ is generated by a non-trivial automorphism of the Dynkin diagram, and the variety $X$ is a Hermitian quadric over $F_G$, i.e., a variety of isotropic lines in the respective Hermitian space (see \cite{SZ10}). By \cite[Thm.~C]{SZ10}, the zero-dimensional part of $\MM(X)$ consists of (shifted) motives of $\Spec F_G$. 

For further properties and examples of non-inner and non-quasi-split hypersurfaces (unitary projective spaces corresponding to algebras with involutions of the second kind), we refer to \cite{Ka12}.
Observe that all these results are in line with the main result of \cite{CM06}, which implies that the zero-dimensional part of the motive of a $G$-flag variety is generated by motives of subfields of $F_G$ containing $F$. 

\subsection{\it Hyperplane sections} 
In our paper we consider only inner hypersurfaces. From this point on we assume $F$ contains a primitive $n$-th root of unity $\zeta$, $n\in F^\times$, and $X$ is an inner hypersurface over $F$.

We introduce our main object, a smooth hyperplane section $Y$ of $X$ as follows. We choose ${a}\in F^\times$ so that $L=F[\sqrt[n]{{a}}]$ is a cyclic Galois extension of degree $n$. Given $b\in F^\times$ we consider a cyclic algebra 
\[
A=F\langle u,v\rangle\big/\left<
u^n={a},\, v^n={b},\,uv=\zeta vu\right>.
\]
This is a central simple algebra over $F$ of degree $n$ which splits over $L$. The $G$-flag variety $X$ for $G=\Aut_F(A)$ is then given by
\[
X=\left\{I_1\subset I_{n-1}\subset A
\colon 
\text{$I_i$ is a right ideal of $A$ of reduced dimension $i$}
\right\}.
\]
We define its smooth hyperplane section $Y$ to be
\[
Y=\left\{(I_1\subset I_{n-1})\in X\colon u I_1\subset I_{n-1}\right\}.
\]
Observe that $Y$ is not a twisted flag variety (it is not a twisted form of a projective homogeneous variety). Our major result says that (up to `phantoms', i.e., up to motives $M$ such that $M_L=0$)

\begin{thmmain}
$\MM(Y) =
\MM(\SB(A))\oplus 
\MM(\SB(A))(1)\oplus \cdots \oplus \MM(\SB(A))(n-3)\oplus \MM(\Spec L)(n-2)$.
\end{thmmain}

\begin{exmain}
The following is the diagram of motivic decomposition of $X$ and $Y$ for $n=5$. 
\[
\begin{array}{c}
\xymatrix@!=0pc{
&&&\bullet\ar@{-}[r]&\bullet\ar@{-}[dr]\\
&&\bullet\ar@{-}[r]&\bullet\ar@{-}[r]&\bullet\ar@{-}[dr]&\bullet\ar@{-}[dr]\\
&\bullet\ar@{-}[r]&\bullet\ar@{-}[r]&\bullet\ar@{-}[r]&\bullet\ar@{-}[dr]&\bullet\ar@{-}[dr]&\bullet\ar@{-}[dr]
\\
\bullet\ar@{-}[r]&\bullet\ar@{-}[r]&\bullet\ar@{-}[r]&\bullet\ar@{-}[r]&\bullet&\bullet&\bullet&\bullet
}\\\\
\MM(X)
\end{array}
\qquad 
\begin{array}{c}
\xymatrix@!=0pc{
&&&
\pentagon
&\\
&&\bullet\ar@{-}[r]&\bullet\ar@{-}[r]&\bullet\ar@{-}[dr]\\
&\bullet\ar@{-}[r]&\bullet\ar@{-}[r]&\bullet\ar@{-}[r]&\bullet\ar@{-}[dr]&\bullet\ar@{-}[dr]
\\
\bullet\ar@{-}[r]&\bullet\ar@{-}[r]&\bullet\ar@{-}[r]&\bullet\ar@{-}[r]&\bullet&\bullet&\bullet
}\\\\
\MM(Y)
\end{array}
\]
Here each chain corresponds to the motive of Severi-Brauer variety $\SB(A)$, and $\pentagon$ corresponds to the Artin motive $\MM(\Spec L)$; each vertex depicts the Tate motive in the decomposition of $\MM(\SB(A))$ over $L$. 
\end{exmain}

\begin{remmain}
1. This result looks quite surprising, as contrary to $Y$ the motive of $X$ does not contain Artin motives. 
Moreover, if $A$ splits (e.g. $b\in \operatorname{Nrd}(A)$), then $\MM(Y)$ consists of Tate motives and a unique Artin motive in the middle dimension:
\[
\MM(Y)=\big(\!\oplus_*\!\ZZ(*)\big)\oplus \MM(\Spec L)(n-2).
\]

2. The main result of \cite{Se08} provides an example of motivic decomposition of a hyperplane section of generalized Severi-Brauer variety $\SB_3(A)$, where $A$ has degree $6$. This decomposition also does not contain Artin motives. Observe that as in \cite{Se08} our result holds modulo `phantoms', since the Rost nilpotence (see e.g. \cite{Br05}) is not known for such varieties.

3. For $n=2$ we just have $Y=\Spec L$. 
For $n=3$ the variety $Y$ is a (twisted) del Pezzo surface of degree 6 (see \cite{Bl10}). 
In this case, our decomposition $\MM(Y)=\MM(\SB(A)) \oplus \MM(\Spec L)(1)$ can be also seen from Manin's blow-up formula \cite[\S9]{Ma68} (which only works for $n=3$). If $A$ splits, this is also a particular case studied in \cite{Gi15}.

4. Using the Artin-Schreier theory, one can also extend it to perfect fields of positive characteristic $p$ and $\deg A=p$.
\end{remmain}

The proof and the paper is organized as follows. In section~\ref{sec:SB} we first discuss properties of the hypersurface $X$ and its hyperplane section $Y$. Then using the classical motivic projective bundle theorem of \cite[\S7]{Ma68} and the results of \cite[\S2]{CPSZ} we split off the Severi-Brauer parts \eqref{eq:MX} and \eqref{eq:MYE}. In the next section~\ref{sec:Galois} we introduce our key tool -- a non-trivial action of the Galois group on the equivariant Chow group $\CH^{n-2}_{T_L}(Y_L)$ (we call it a monodromy action). By \cite{BP22} since $X_0$ is the adjoint variety for $\PGL_n$, the section $Y_L$ is $T_L$-stable, where $T$ is a maximal torus of $\PGL_n$. This allows us to describe the equivariant Chow groups of $Y_{L}$ and the monodromy action (see Proposition~\ref{thm:mainmon}) in terms of the $T$-fixed point locus of $X_0$. Finally, in section~\ref{sec:Artin} using the monodromy action and the localization techniques of \cite{BP22} and \cite{Br97}
we construct an explicit idempotent defining the Artin motive (see Proposition~\ref{prop:art}).


\section{The Severi-Brauer part}\label{sec:SB}

\subsection{\it The motive of the inner Milnor hypersurface}
It was proven in \cite[Cor.2.3]{CPSZ} that the motive of inner hypersurface $X$ decomposes into shifted copies of the motive of Severi-Brauer variety $\SB(A)$. Below we describe this splitting explicitly following the motivic projective bundle theorem of~\cite[\S7]{Ma68}.

\smallskip

Consider projection $\pi\colon X\to \SB(A)$ which sends $(I_1\subset I_{n-1})$ to $I_1$. Denote by $\II$ the rank $n$ tautological bundle over $\SB(A)$. Denote by $\Aa$ the constant algebra sheaf over $\SB(A)$. By \cite[Lem.4.4, Lem.4.5, and Rem.4.6]{CPSZ} we can identify $\pi\colon X\to\SB(A)$ with the Grassmann bundle $\Gr(n-2,\FF)\to \SB(A)$, where $\FF = (\Aa/\II) \otimes_\Aa \II^\vee$ is a bundle of rank $n-1$ over $\SB(A)$. After dualizing $\FF$, we may identify $\pi$ with the projective bundle $\PP(\FF^\vee)\to \SB(A)$, where 
\[
\FF^\vee=\II\otimes_\Aa (\Aa/\II)^\vee=\mathcal{H}om_{\Aa}(\Aa/\II,\II).
\]
Observe that $\FF^\vee \times_F \BF=\OO(-1)\otimes (\OO^{\oplus n}/\OO(-1))^\vee$ over $\PP^{n-1}$. Consider a chain of tautological bundles $\II_1\subset \II_{n-1}\subset \Aa_X=\pi^*(\Aa)$ over $X$. Define a line bundle over $X$
\[
\LL=\Aa_X/\II_{n-1}\otimes_{\Aa_X}\II_1^\vee =\mathcal{H}om_{\Aa_X}(\II_1,\Aa_X/\II_{n-1}).
\]

\begin{lem}
We have $\LL=\OO(1)$ over $X= \PP(\FF^\vee)$.
\end{lem}

\begin{proof}
The tautological bundle over $\Gr(n-2,\FF)$ is $\II_{n-1}/\II_1\otimes_{\Aa_X}\II_1^\vee$ and the quotient bundle is given by $\Aa_X/\II_{n-1}\otimes_{\Aa_X} \II_1^\vee$. The latter then corresponds to the bundle $\OO(1)$ over $\PP(\FF^\vee)$ by \cite[Ex.14.6.5]{Fu98}.
\end{proof}
 
Consider the first Chern class $H=c_1(\LL)\in \CH^1(X)$ in the Chow group of algebraic cycles modulo rational equivalence relation (see e.g. \cite{Fu98}). By the projective bundle theorem, $\CH(X)$ is freely generated by the classes $1,H,\ldots,H^{n-2}$ as a $\CH(\SB(A))$-module (via $\pi^*$). In other words, for any $\gamma\in \CH^m(X)$ we have a presentation 
\[
\gamma=\pi^*(\gamma_0)+\pi^*(\gamma_1)H+ \ldots+ \pi^*(\gamma_{n-2})H^{n-2}\;\text{ for unique }\gamma_i\in  \CH^{m-i}(\SB(A)).
\]

Following \cite[\S7]{Ma68} (for $H=x$, $n-2=r$, $f_i=g_i$ loc.cit.) we first construct correspondences 
\[
f_i \in \CH^{2n-3-i}(X\times \SB(A))=\Hom^{-i}(\MM(X),\MM(\SB(A))),\; i=0,\ldots, n-2
\] 
by downward induction on $i$. We set $f_{n-2} \in \CH^{n-1}(X\times \SB(A))$ to be the transposed class of the graph of $\pi$, and if $f_{i+1},f_{i+2},\ldots,f_{n-2}$ are already defined we set
\[
f_i= f_{n-2}\circ c_{H^{n-2-i}}\circ \Big(\Delta_X - \sum_{k=i+1}^{n-2} c_{H^k}\circ f_k\Big),\;\text{ where }c_\gamma=\Delta_{X*}(\gamma)\in \CH^{2n-3+m}(X\times X).
\]
Such correspondences satisfy $f_i\circ \gamma=\gamma_i$ for each $i$. We then set 
\[
g_i=c_{H^i}\circ f_{n-2}^t \in \CH^{n-1+i}(\SB(A)\times X)=\Hom^i(\MM(\SB(A)),\MM(X)), \; i=0,\ldots,n-2.
\]  
By definition, $f_i\circ g_j=\delta_{i,j}\Delta_{\SB(A)}$ (here $\delta_{i,j}$ is the Kronecker symbol). The correspondences 
\[
p_{n-2-i}=g_i \circ f_i,\quad 0\le i \le n-2
\] 
form a complete system of pair-wise orthogonal idempotents which gives the direct sum decomposition of \cite[Cor.2.3]{CPSZ}:
\begin{equation}\label{eq:MX}
\MM(X)=\oplus_{i=0}^{n-2} (X,p_i)=\MM(\SB(A))\oplus \MM(\SB(A))(1)\oplus \ldots \oplus \MM(\SB(A))(n-2).
\end{equation}


\subsection{\it The Severi-Brauer part of the hyperplane section}
Using properties of the first Chern class we show that the complete system of idempotents defining~\eqref{eq:MX} can be restricted
to an incomplete system of idempotents for $Y$, hence, providing a partial motivic decomposition~\eqref{eq:MYE}.

\begin{lem}
The variety $Y=\left\{(I_1\subset I_{n-1})\in X \colon u I_1\subset I_{n-1}\right\}$ is smooth. 
\end{lem}

\begin{proof}
It suffices to check this over $\BF$. By the very definition we have
\[
Y(\BF)=\{(x,y)\in \BF^n\times \BF^n\colon x\neq 0,\, y\neq 0,\, x u y^t=xy^t=0\}/\BF^\times \times \BF^\times,
\]
where $x$ and $y$ are row vectors, and $u$ is a $n\times n$-matrix. Note that the roots $\zeta^i\sqrt[n]{a}$ of $z^n-{a}$ give distinct eigenvalues of $u$. After canceling $\sqrt[n]{a}$, conditions $x u y^t=xy^t=0$ turn into
\[
\begin{cases}
    x_1y_1+x_2y_2+\cdots+x_ny_n = 0\\
    x_1y_1+\zeta x_2y_2+\cdots+\zeta^{n-1}x_ny_n =0.
\end{cases}
\]
Consider the Jacobi matrix 
\[
\begin{bmatrix}
y_1 & x_1 & y_2 & x_2 & \cdots&\cdots & y_n & x_n\\
y_1 & x_1 & \zeta y_2 & \zeta x_2 &\cdots &\cdots & \zeta^{n-1}y_n & \zeta^{n-1} x_n
\end{bmatrix}.
\]
Since $x\neq 0$, $y\neq 0$ and $xy^t=0$, there exist $x_i\neq 0$ and $y_j\neq 0$ for $i\neq j$. Since  submatrix 
\[
\begin{bmatrix}
x_i & y_j\\
\zeta^{i-1} x_i & \zeta^{j-1} y_j
\end{bmatrix}
\]
has rank 2, then so is the Jacobi matrix. Hence, the variety
\[
\{(x,y)\in \BF^n\times \BF^n\colon x\neq 0,\, y\neq 0,\, x u y^t=xy^t=0\}
\]
is smooth. Since it is a $\mathbb{G}_m\times \mathbb{G}_m$-principal bundle over $Y$, $Y$ is also smooth.
\end{proof}

\begin{lem} 
We have $[Y]=H$ in $\CH^1(X)$. 
\end{lem}

\begin{proof}
The left multiplication $I_1 \to A \xrightarrow{u\cdot} A \to A/I_{n-1}$ defines a zero-section in the line bundle $\LL$.
\end{proof}

Let $\imath\colon Y \hookrightarrow X$ denote the closed embedding. Consider the induced pull-back $\imath^*\colon \CH(X\times \SB(A) \to \CH(Y\times \SB(A))$. For $0\leq i\leq n-2$ we set
\[
\overline{f}_i =\imath^*(f_i) \in \CH^{2n-3-i}(Y\times \SB(A))\quad\text{ and }\quad\overline{g}_i =\imath^*(g_i) \in \CH^{n-1+i}(\SB(A) \times Y).
\]

\begin{lem} 
For $0\leq i,j \leq n-3$ we have $\overline{f}_{i+1}\circ \overline{g}_{j}=\delta_{i,j}\Delta_{\SB(A)}$.
\end{lem}

\begin{proof}
Let $p_{ij}\colon X_1\times X_2\times X_3 \to X_i\times X_j$, $i<j$  denote the projection, and let $p_{ij}^*$, $p_{ij*}$ denote the respective pull-back and push-forward. Applying the projection formula of \cite{Ma68}  we obtain
\begin{align*}
\overline{f}_{i+1}\circ \overline{g}_{j}
&=\bar p_{13*}(\bar p_{23}^*\overline{f}_{i+1} \cdot \bar p_{12}^*\overline{g}_{j}) =\bar p_{13*}(\bar p_{23}^*\imath^*f_{i+1} \cdot \bar p_{12}^*\imath^*g_{j})\\
&=p_{13*}\imath_* (\imath^*p_{23}^*f_{i+1} \cdot \imath^*p_{12}^*g_{j})=p_{13*}\imath_*(\imath^*(p_{23}^*f_{i+1} \cdot p_{12}^*g_{j})\cdot 1)\\
&=p_{13*}(p_{23}^*f_{i+1} \cdot p_{12}^*g_{j}\cdot \imath_*(1))=p_{13*}(p_{23}^*f_{i+1} \cdot p_{12}^*(g_{j}\cdot \imath_*(1))). 
\end{align*}
But $\imath_*(1)\cdot g_j=c_{H}\circ g_j=g_{j+1}$ by \cite[Lemma, p.449]{Ma68}. Therefore, we get
\[
\overline{f}_{i+1}\circ \overline{g}_{j}=f_{i+1}\circ g_{j+1}=\delta_{i+1,j+1}\Delta_{\SB(A)}.\qedhere
\]
\end{proof}

By the lemma, the correspondences 
\[
\overline{p}_i=\overline{g}_i\circ \overline{f}_{i+1} \in \CH^{2n-4}(Y\times Y),\quad 0\leq i\leq n-3
\]
form an orthogonal (incomplete) system of idempotents which gives the following direct sum decomposition:
\begin{equation}\label{eq:MYE}
\MM(Y)=\big(\oplus_{i=0}^{n-3}(Y,\overline{p}_i)\big)\oplus \operatorname{im}e_0=\MM(\SB(A))\oplus 
\MM(\SB(A))(1)\oplus \ldots \oplus \MM(\SB(A))(n-3)\oplus \operatorname{im}\overline{p},
\end{equation}
where $\overline{p}=\Delta_{Y}-\sum_{i=0}^{n-3}\overline{p}_i$ is some nonzero idempotent.


\section{The monodromy action}\label{sec:Galois}
 
\subsection{\it The Galois action and the correspondence product}
Denote $\Ga_L=\Gal(L/F)$. By definition of the Galois extension, we have $L\otimes L\xrightarrow{\simeq} \prod_{\sigma\in \Ga_L}\!\! {}^\sigma\! L$. So for any smooth projective variety $Z$ over $F$ there is a commutative diagram
\[
\xymatrix{
Z_L \times \Spec L \ar[d] & \coprod_{\sigma\in \Ga_L} Z\times \Spec L \ar[l]_-{\simeq} \\
Z\times \Spec L  & Z\times \Spec L \ar[u]_{\imath_\sigma} \ar[l]_{\id \times \sigma}
},
\]
where $\imath_\sigma$ identifies $Z_L$ with the respective $\sigma$-component. 
Taking the induced pull-backs we obtain the following formula for the restriction map
\begin{equation}\label{eq:restr}
\res_{L/F}\colon \CH(Z\times \Spec L) \to \CH(Z_L \times \Spec L) \xrightarrow{\simeq}
\bigoplus_{\sigma \in \Ga_L} \CH(Z\times \Spec L), \quad
\res_{L/F}(\gamma)=(\sigma \gamma)_{\sigma\in \Ga_L},
\end{equation}
where $\sigma \in \Ga_L$ acts on $\CH(Z\times \Spec L)$ via the second factor.

\begin{lem}\label{lem:corrg}
For $\gamma_1\in \Hom(\MM(\Spec L),\MM(Z))$ and $\gamma_2\in \Hom(\MM(Z),\MM(\Spec L))$ we have
\[
\res_{L/F}(\gamma_1\circ\gamma_2)= \sum_{\sigma\in \Ga_L} \sigma \gamma_1\boxtimes \sigma \gamma_2 \quad\text{ and }\quad
\res_{L/F}(\gamma_2\circ\gamma_1)=(\langle \sigma \gamma_2, \tau \gamma_1\rangle_{Z_L})_{\sigma,\tau \in \Ga_L},
\]
where $\langle \gamma_1,\gamma_2\rangle_{Z_L}$ denotes the $L$-linear pairing $(\gamma_1,\gamma_2)\mapsto p_*(\gamma_1\cdot \gamma_2)$, $p\colon Z_L \to \Spec L$ is the structure map.
\end{lem}

\begin{proof} 
Use $\res_{L/F}(\gamma_i\circ\gamma_j) =\res_{L/F}(\gamma_i)\circ\res_{L/F}(\gamma_i)$, $\MM(\Spec L\otimes L)\simeq \oplus_{\sigma\in \Ga_L} \MM(\Spec L)$ and apply \eqref{eq:restr}.
\end{proof}

\begin{cor}\label{cor:circ}
We have
$\gamma_2\circ\gamma_1=\Delta_{\Spec L} \Longleftrightarrow \langle \gamma_1, \sigma\gamma_2 \rangle_{Z_L}=\delta_{\id,\sigma}$.
\end{cor}

 
\subsection{\it The $1$-cocycle}

Following \cite[\S2.5]{GS06} there is an explicit isomorphism $\rho\colon A_L\to M_n(L)$ of $L$-algebras given by 
\[
\rho(u)=\begin{bmatrix}
\sqrt[n]{a}\\
& \zeta \sqrt[n]{a}\\
&& \ddots \\
&&& \zeta^{n-1}\sqrt[n]{a}
\end{bmatrix},\qquad
\rho(v)=\begin{bmatrix}
0&&&b\\
1&\ddots\\
&\ddots&0 \\
&&1&0
\end{bmatrix}.
\]
Observe that $\rho$ does not commutes with obvious Galois group actions on $A_L$ and $M_n(L)$. The obstruction is given by a $1$-cocycle
\[
\mathfrak{a}\colon \Ga_L\longrightarrow \Aut_L(M_n(L))\simeq \PGL_n(L),
\qquad \sigma\longmapsto \mathfrak{a}_{\sigma}
=\rho\circ \sigma\circ \rho^{-1}\circ\sigma^{-1}.
\]
We can describe this cocycle explicitly as follows
\begin{lem} For the cyclic generator $\eta \in \Ga_L$ with $\eta(\sqrt[n]{a})=\zeta \sqrt[n]{a}$ we have 
\[
\mathfrak{a}_{\eta^k}=M^k,\quad k=0,\ldots,n-1,\] where $M$ denotes the class of $\rho(v)$ in  $\PGL_n(F)$.
\end{lem}

\begin{proof}
Indeed, for $z=\eta^k \rho(u)=\zeta^k \rho(u)$ and $z=\eta^k \rho(v)=\rho(v)$ we have
\[
\zeta^k M^k\rho(u)M^{-k} = \rho(u)\quad \text{ and }\quad
M^k\rho(v)M^{-k} = \rho(v).
\]
So that $M^kzM^{-k} = \rho(\eta^k(\rho^{-1}(\eta^{-k} z)))$. 
\end{proof}

Let $T\subset \PGL_n(F)$ be the maximal torus (the subgroup of diagonal matrices). 
Since $M$ normalizes $T$, we immediately obtain the following 
\begin{cor} The $1$-cocycle $\mathfrak{a}$ gives a group homomorphism
\[
\mathfrak{a}\colon \Ga_L\longrightarrow N_T(\PGL_n(F)),\quad \eta\mapsto \mathfrak{a}_\eta=M.
\]
\end{cor}
In this way, we obtain a group homomorphism from $\Ga_L$ to the symmetric (Weyl) group $S_n$. We will denote by the same symbol $\eta$ (resp. $\sigma$) the image of an element of $\Ga_L$ in $S_n$.

\subsection{\it The $T$-fixed point locus} 
Consider the Milnor hypersurface $X_0=\{V_1\subset V_{n-1} \subset F^n\}$. The torus $T\subset \PGL_n(F)$ acts by conjugations on $M_n(F)$ and, hence, by scaling basis vectors $\{e_1,\ldots,e_n\}$ of $F^n$. The subset $X_0^T$ of $T$-fixed points then consists of flags
\[
[ij]=(\Span(e_i)\subset \Span(e_1,\ldots,\cancel{e}_{j},\ldots,e_n))\in X_0, \quad 1\le i\neq j\le n.
\]
Consider the induced isomorphism $\rho\colon X_L\xrightarrow{\simeq} (X_0)_L$ together with the induced action by $T_L$ on $X_L$. 

\begin{lem}
We have the following commutative diagram of $T_L$-fixed point loci:
\[
\xymatrix{
 (X_L)^{T_L}\ar[r]^-\rho\ar[d]_{\sigma} & X_0^T\times \Spec L \ar[d]^{\mathfrak{a}_\sigma\times \sigma} \ar@{=}[r] & {\coprod}_{1\leq i\neq j\leq n} [ij]_L \ar[d]^{\mathfrak{a}_\sigma\times \sigma} \\
 (X_L)^{T_L}\ar[r]^-\rho & X_0^T\times \Spec L \ar@{=}[r] & {\coprod}_{1\leq i\neq j\leq n} [ij]_L.
}
\]
\end{lem}

\begin{proof}
Indeed, $\mathfrak{a}_\eta=M$ acts on the basis  as follows: $e_1\mapsto e_2$, $e_2\mapsto e_3$, $\ldots$, $e_n\mapsto be_1$. So it maps the flag $[ij]$ to the flag $[\eta(i)\eta(j)]$, where $\eta=(1,2,\ldots,n)\in S_n$ is the corresponding cyclic permutation. Therefore, $\mathfrak{a}_\sigma\times\sigma$ is an automorphism of the $T_L$-fixed locus $X_0^T\times\Spec L$.
\end{proof}

Observe that $\mathfrak{a}_\sigma$ does not commute with the $T$-action on $X_0$ but it commutes up to a conjugation.
Namely, let $\operatorname{Inn}(\mathfrak{a}_\sigma)(t)=\mathfrak{a}_\sigma t \mathfrak{a}_\sigma^{-1}$, $t\in T$. Then there is a commutative diagram
\begin{equation}\label{diag:Tact}
\xymatrix{
T\times X_0 \ar[r] \ar[d]_{\operatorname{Inn}(\mathfrak{a}_\sigma) \times \mathfrak{a}_\sigma} & X_0 \ar[d]^{\mathfrak{a}_\sigma}\\
T\times X_0 \ar[r] & X_0.
}
\end{equation}

\subsection{\it The Galois action on equivariant Chow groups} Consider the $T$-equivariant Chow group $\CH_{T}(X_0)$ of \cite{To97,Br97}.
For the $T$-fixed point locus we obtain
\[
\CH_T(X_0^T)=\CH_T\big({\textstyle \coprod}_{1\leq i\neq j\leq n} [ij]\big)\simeq \bigoplus_{1\leq i\neq j\leq n}\Sym_{\ZZ}T^*,
\]
where $T^*=\Hom(T,F^\times)$ is the group of characters of $T$.
According to \cite[(3.2),(3.3)]{Br97} the pull-back induced by the embedding $X_0^T \hookrightarrow X_0$ gives the inclusion $\CH_T(X_0) \hookrightarrow \CH_T(X_0^T)$. Combining all these observations together we obtain

\begin{prop} \label{thm:mainmon}
There is a commutative diagram where 
$\widehat \sigma (\varphi_{ij}\big)_{ij} = (\sigma\,\varphi_{\sigma^{\text{-}1}(i)\sigma^{\text{-}1}(j)})_{ij}$, $\varphi_{ij}\in \operatorname{Sym}_{\mathbb{Z}} T^*$
\[
\xymatrix{
 \CH_{T_L}(Y_L) \ar[d]_{\sigma}  & \CH_{T_L}(X_L) \ar[d]_{\sigma}  \ar[l]_{\imath^*} & \CH_{T_L}(X_0 \times \Spec L) \ar@{^(->}[r]\ar[l]_-{\rho^*} \ar[d]^{\mathfrak{a}_\sigma \times\sigma}& \bigoplus_{1\leq i\neq j\leq n}\operatorname{Sym}_{\mathbb{Z}}T^* \ar[d]^{\widehat \sigma}\\
 \CH_{T_L}(Y_L)   &  \CH_{T_L}(X_L)  \ar[l]_{\imath^*} & \CH_{T_L}(X_0 \times \Spec L) \ar@{^(->}[r] \ar[l]_-{\rho^*} & \bigoplus_{1\leq i\neq j\leq n}\operatorname{Sym}_{\mathbb{Z}} T^*
}.
\]
\end{prop}

\begin{proof}
Recall that $\mathfrak{a}_\sigma$ permutes the $T$-fixed points via $\mathfrak{a}_\sigma([ij])=[\sigma(i)\sigma(j)]$.
By \eqref{diag:Tact} it also induces an endomorphism of $\CH_T([ij])$ given by the action of the Weyl (symmetric) group 
via $\lambda \circ \operatorname{Inn}(\mathfrak{a}_\sigma)^{-1}= \sigma \lambda$, $\lambda\in T^*$ (cf. \cite{EG97}).
Observe that the Galois group acts trivially on $T^*$ \cite[\S13.2]{Sp09}.

Finally, since the matrix $\rho(u)$ is diagonal, the image $\rho(Y_L)$ is $T_L$-stable and, moreover, $(Y_L)^{T_L}=(X_L)^{T_L}$ (see also \cite{BP22}). Therefore, we can restrict the action by $\sigma$ to $\CH_{T_L}(Y_L)$.
\end{proof}

\begin{rem}
Observe that the action by $\sigma$ is trivial on $\CH(X_L)$. It is also trivial on all Chow groups  $\CH^i(Y_L)$ except for $i=n-2$ (the middle dimension). We call the action by $\sigma$ on the equivariant Chow groups of $Y_L$ and $X_L$ (resp. by $\widehat \sigma$ on $\CH_{T_L}((X_0^{T})_L)$) the monodromy action.
\end{rem}

\section{The Artin motive part}\label{sec:Artin}

\subsection{\it The localization for $Y_L$}
We consider a graph with $n(n-1)$ vertices denoted $[ij]$, $1\leq i\neq j\leq n$, which has two types of labelled edges
\[
[ij]\stackrel{\alpha_{jk}}{\text{------}}[ik] 
\quad \text{ and }\quad 
[ij]\stackrel{\alpha_{ik}}{\text{------}}[kj],
\quad \text{ where all $i,j,k$ are distinct,}
\]
and $\alpha_{ij}=t_i-t_j\in T^*$, $t_i\colon diag(z_1,\ldots,z_n)\mapsto z_i^{-1}$.

\begin{lem} 
The image of $\CH_{T_L}(Y_L)$ in $\oplus_{1\leq i\neq j\leq n}\Sym_{\mathbb{Z}} T^*$ under the horizontal maps of the diagram of Theorem~\ref{thm:mainmon} is given by
\[
\CH_{T_L}(Y_L)\simeq \big\{
(\varphi_{ij})_{ij}: 
\alpha|\varphi_{ij}-\varphi_{kh}
\text{ for any edge $[ij]\stackrel{\alpha}{\text{---}}[kh]$}
\big\}.
\]
Moreover, the Poincar\'e pairing is given by
\begin{equation}\label{eq:poinc}
\langle \varphi,\psi\rangle_{Y_L}
=\sum_{1\leq i\neq j\leq n}\frac{\varphi_{ij}\psi_{ij}}{\prod_{s\neq i,j}\alpha_{is}\alpha_{sj}}.
\end{equation}
\end{lem}

\begin{proof}
The first part of the lemma is a consequence of results of \cite[\S3]{BP22} and \cite[\textsection 3.4]{Br97}.
Namely, by~\cite[Def.3.6]{BP22} there are three types of $T$-stable curves $C$ over $X_0$ (the weight of $C$ at a $T$-fixed point $x$ is the $T$-weight in $\Tan_{x}C$):
\begin{itemize}
\item[(i)] 
a root-conic curve connecting $[ij]$ and $[ji]$ with weight $\alpha_{ij}$ at $[ij]$:
\[
\PP^1\ni [x\colon y]\mapsto
\big(\Span(xe_i+ye_j)\subset \Span(e_1,\ldots,\cancel{e}_{i},\ldots,\cancel{e}_{j},\ldots,e_n,xe_i+ye_j)\big),
\]
\item[(ii)] 
a plane curve connecting $[ij]$ and $[ik]$ with weight $\alpha_{kj}$ at $[ij]$ (for distinct $i,j,k$): 
\[
\PP^1\ni [x\colon y]\mapsto
\big(\Span(e_i)\subset \Span(e_1,\ldots,\cancel{e}_{j},\ldots,\cancel{e}_{k},\ldots,e_n,ye_j+xe_k)\big),
\]
\item[(iii)] 
a plane curve connecting $[ij]$ and $[kj]$ with weight $\alpha_{ik}$ at $[ij]$ (for distinct $i,j,k$): 
\[
\PP^1\ni [x\colon y]\mapsto
\big(\Span(xe_i+ye_k)\subset \Span(e_1,\ldots,\cancel{e}_{j},\ldots,e_n)\big).
\]
\end{itemize}        
By Prop.~3.8~loc cit., only the plane curves are contained in $Y_L$ (over $L$). 
Observe that weights of all edges are primitive (not divisible in $\Sym_{\ZZ} T^*$). 
The description of $\CH_{T_L}(Y_L)$ then follows from \cite[\textsection 3.4]{Br97}. 

As for the Poincar\'e pairing, observe that there is an inclusion $\Tan_{x}C_L\subset \Tan_{x}Y_L$
for each $x=[ij]\in C$. Recall that the $T$-weights of $\Tan_x C_L$ are $\alpha_{is}$ and $\alpha_{sj}$ ($s\neq i,j$)
which are all distinct. By dimension reasons we conclude that
$\Tan_{x}Y_L=\oplus_{C\ni x} \Tan_{x}C_L$. So the equivariant Euler class is given by
$e^{T}(\Tan_x Y_L)=\prod_{s\neq i,j}\alpha_{is}\alpha_{sj}$,
and the desired formula then follows.
\end{proof}

\subsection{\it The Artin motive support} 
For $1\leq \ell\leq n$ we define the following cycles in $\oplus_{1\leq i\neq j\leq n}\Sym_{\mathbb{Z}} T^*$
\[
\gamma_\ell = (\gamma^{\ell}_{ij})_{ij},\;\text{ where }\gamma^{\ell}_{ij}=\delta_{i,\ell}{\textstyle \prod}_{s\neq i,j}\alpha_{is}.
\]

\begin{lem} \label{lem:gammad} 
We have 
{\rm (i)}~$\gamma_{\ell}\in \CH_T^{n-2}(Y_L)$,\ \ 
{\rm (ii)}~$\widehat \sigma\gamma_{\ell}=\gamma_{\sigma(\ell)}$, $\sigma\in \Ga_L$,\ \  
and~{\rm (iii)}~$\langle \gamma_{k},\gamma_{\ell}\rangle_{Y_L}=(-1)^{n-2}\delta_{k,\ell}$.
\end{lem}

\begin{ex}
For $n=5$ there are $20$ vertices which can be distributed as vertices of a regular dodecahedron. Two types of edges then form two families (compounds) of $5$ tetrahedra:
\[    
\href{https://cubicbear.github.io/Compounds}
{\includegraphics[width=0.5\linewidth]{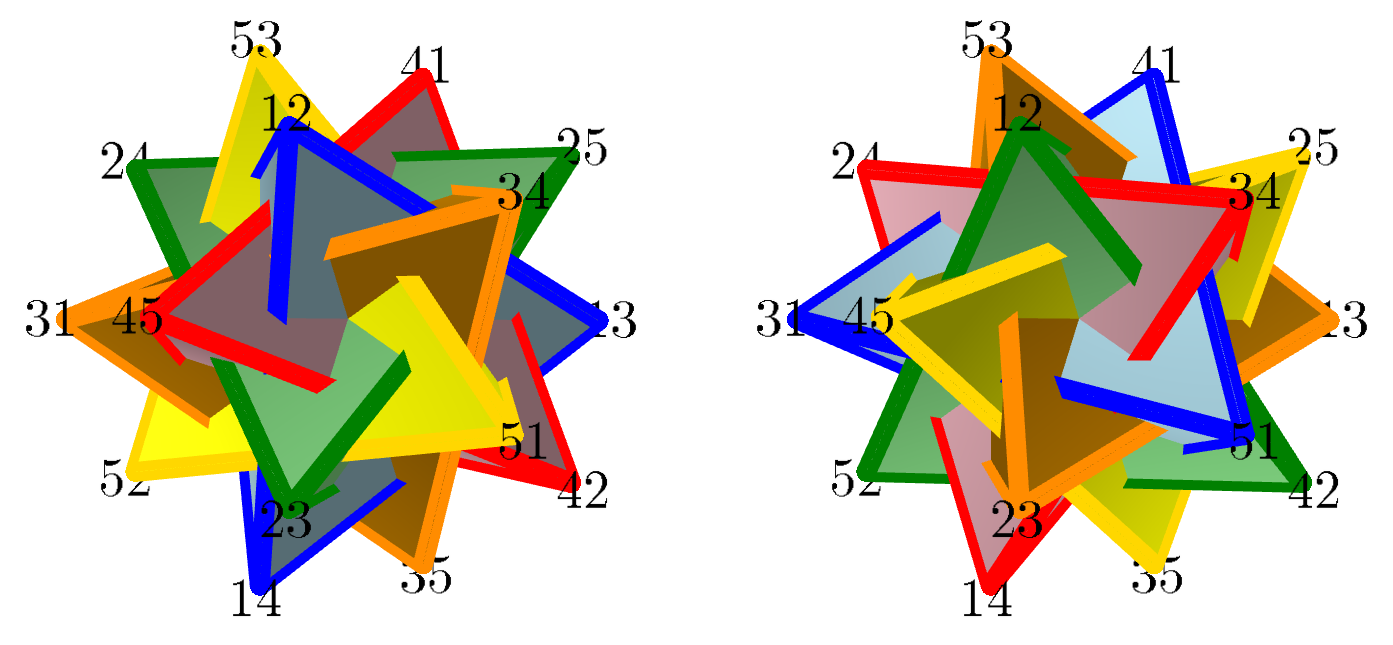}}
\]
Each $\gamma^{\ell}$, $1\leq \ell\leq 5$ is supported over a single tetrahedron on the left diagram, and the monodromy action is given by rotations. 
\end{ex}

\begin{proof}
(i) For the edge $[ij]\stackrel{\alpha_{jk}}{\text{---}}[ik]$ we have 
\[
\alpha_{jk}\mid \delta_{i,\ell}\Big({\textstyle \prod}_{s\neq i,j}{\alpha_{is}}-{\textstyle \prod}_{s\neq i,k}{\alpha_{is}}\Big)
=\delta_{i,\ell} (\alpha_{ik}-\alpha_{ij}){\textstyle \prod}_{s\neq i,j,k}\alpha_{is}=\delta_{i,\ell} \alpha_{jk}{\textstyle \prod}_{s\neq i,j,k}\alpha_{is}.
\]
For the edge $[ij]\stackrel{\alpha_{ik}}{\text{---}}[kj]$ we have 
$\alpha_{ik}\mid \delta_{i\ell}\prod_{s\neq i,j}\alpha_{is}-\delta_{k\ell}\prod_{s\neq k,j}\alpha_{ks}$, since $\alpha_{ik}$ divides both terms. 

(ii) We have
$\sigma\gamma^{\ell}_{\sigma^{-1}(i)\sigma^{-1}(j)}=\delta_{\sigma^{-1}(i),\ell}\prod_{\sigma(s)\neq i,j}\alpha_{i\sigma(s)}=\gamma^{\sigma(\ell)}_{ij}$.

(iii) Since $\gamma_k\gamma_{\ell}=0$ for $k\neq \ell$, we may assume $k=\ell$. 
By \eqref{eq:poinc} we have
\begin{align*}
\langle \gamma_{\ell},\gamma_{\ell}\rangle_{Y_L}
&=
\sum_{j\neq \ell}\frac{\prod_{s\neq \ell,j}\alpha_{\ell s}^2}{\prod_{s\neq \ell,j}\alpha_{\ell s}\alpha_{sj}}
=
\sum_{j\neq \ell}\frac{\prod_{s\neq \ell,j}\alpha_{\ell s}}{\prod_{s\neq \ell,j}\alpha_{sj}}.
\end{align*}
Consider this expression as a polynomial $f(t)$ in $t=t_{\ell}$ of degree at most $n-2$ with coefficients in $\ZZ(t_1,\ldots,\widehat{t_\ell},\ldots,t_n)$. 
Observe that for each $i\neq \ell$ we have
$$f(t_i)=\sum_{j\neq \ell}\frac{\prod_{s\neq \ell,j}\alpha_{i s}}{\prod_{s\neq \ell,j}\alpha_{sj}}
=\frac{\prod_{s\neq \ell,i}\alpha_{is}}{\prod_{s\neq \ell,i}\alpha_{si}}=(-1)^{n-2}.
$$
Then by the Lagrange interpolation $f(t)=(-1)^{n-2}$. 
\end{proof}

Let $\overline{\gamma}_\ell$ denote the image of $\gamma_\ell$ in $\CH^{n-2}(Y_L)$. We set
\[
f =\overline{\gamma}_1\in \CH^{n-2}(Y\times \Spec L)\quad \text{ and }\quad
g = (-1)^{n-2}\overline{\gamma}_1\in \CH^{n-2}(\Spec L \times Y).
\]

\begin{prop}\label{prop:art}
The composite $p=g\circ f$ is an idempotent in $\CH(Y\times Y)$, and $(Y,p)\simeq \MM(\Spec L)(n-2)$. 

Moreover, the motive $(Y,p)$ splits over $L$ into a direct sum of $n$ shifted Tate motives $\ZZ(n-2)$
each supported at $(-1)^{n-2}\overline{\gamma}_\ell$, $1\le \ell \le n$.
\end{prop}

\begin{proof}
By Lemma~\ref{lem:gammad} and Corollary~\ref{cor:circ} we obtain that $f\circ g=\Delta_{\Spec L}$ which implies the first part of the theorem. As for the second part observe that by Lemma~\ref{lem:corrg} we also have
\[
\res_{L/F}(g\circ f)=\sum_{\ell=1}^{n} (-1)^{n-2} \overline{\gamma}_\ell\boxtimes \overline{\gamma}_\ell,
\]
and by Lemma~\ref{lem:gammad} each $(-1)^{n-2} \overline{\gamma}_\ell\boxtimes \overline{\gamma}_\ell$ is an idempotent in $\CH(Y_L\times_L Y_L)$.
\end{proof}

\subsection{\it Orthogonality to the Severi-Brauer part} 
Recall that there is a $T_L$-equivariant line bundle $\LL_L$ over $X_L$ with $H_L=c_1(\LL_L)$. 
Let $\overline{H}=\imath^*(H_L)\in \CH^1(Y_L)$ denote its restriction to $Y_L$.
Recall also that $X_L$ is a $T_L$-equivariant projective bundle $\PP(\FF^\vee_L)$ over $\PP^{n-1}_L$.
Let $\overline{h}=\imath^*h \in \CH^1(Y_L)$ denote the restriction of the pull-back $h\in \CH^1(X_L)$ of the class of a $T$-equivariant hyperplane section in $\PP^{n-1}$.

Applying Manin's projective bundle theorem to the tower of two projective bundles $X_L/\PP^{n-1}_L$, we obtain

\begin{lem}\label{lem:orth}
The image of each $\overline{g}_j$ in $\CH^{n-2}(Y_L)$ is a free $\ZZ$-module with basis 
$\{\overline{h}^{i+1}\overline{H}^j\}_ {i,j\ge 0,\, i+j=n-3}$. Moreover, we have 
\[
\langle \overline{h}^{i+1}\overline{H}^j,\overline{h}^{i'+1}\overline{H}^{j'}\rangle_{Y_L} =\delta_{i+i',n-3}.
\]
\end{lem}

\begin{lem}\label{lem:ortg}
For any $1\leq \ell\leq n$ and $\gamma\in \CH^{n-3}(Y_L)$, we have $\langle \overline{\gamma}_\ell, \overline{h}\gamma\rangle_{Y_L}=0$.
\end{lem}

\begin{proof} Consider liftings $\gamma_\ell$ of $\overline{\gamma}_\ell$, $\gamma'$ of $\gamma$ and $h'=(t_i-t_1)_{ij}$ of 
$\overline{h}$ to  $\CH_{T_L}(Y_L)$. We then have
$
\langle \gamma_\ell,h'\gamma'\rangle_{Y_L}=
\langle h'\gamma_\ell,\gamma'\rangle_{Y_L}=
\langle (t_\ell-t_1)\gamma_\ell,\gamma'\rangle_{Y_L}=
(t_\ell-t_1)
\langle \gamma_\ell,\gamma'\rangle_{Y_L}=0
$.
\end{proof}

\begin{cor}
The classes $\{\overline{\gamma}_\ell\}_{1\leq \ell\leq n}$ and $\{\overline{h}^{i+1}\overline{H}^{j}\}_{i,j\geq 0,\,i+j=n-3}$
form a basis of $\CH^{n-2}(Y_L)$. 
\end{cor}

\begin{proof}
By \cite[Prop.3.19]{BP22} the rank of $\CH^{n-2}(Y_L)$ coincides with the total number of classes that is $2n-2$. 
Since, the pairing $\langle \text{-},\text{-}\rangle_{Y_L}$ on these classes is non-degenerate with determinant $\pm 1$, the result follows.
\end{proof}

\begin{prop}\label{prop:orth}
The idempotent $p$ is orthogonal to each $\overline{p}_j=\overline{g}_j\circ \overline{f}_{j+1}$, $j=0\ldots n-3$ from \eqref{eq:MYE}, and $p=\bar p$ over $L$.
\end{prop}

\begin{proof}
As for the first, by Lemma~\ref{lem:ortg} we have $p_*(\overline{h}^{i+1}\overline{H}^j)=0$. So it then follows by Lemma~\ref{lem:orth}. As for the second, it follows by comparing the respective ranks.
\end{proof}

Finally, combining the decomposition \eqref{eq:MYE} with Propositions~\ref{prop:art} and \ref{prop:orth} we obtain the motivic decomposition theorem of the introduction.


\begin{thebibliography}{99}

\bibitem[Bl10]{Bl10}
Blunk,~M.
Del Pezzo surfaces of degree 6 over an arbitrary field.
{\it J. Algebra} {\bf 323} (2010), 42--58.

\bibitem[BP22]{BP22}
Benedetti,~V.; Perrin,~N. 
Cohomology of hyperplane sections of (co)adjoint varieties. 
{\it arXiv:2207.02089}.

\bibitem[Br97]{Br97}
Brion,~M.
Equivariant Chow groups for torus actions.
{\it Transform. Groups} {\bf 2} (1997), no. 3, 225--267.

\bibitem[Br05]{Br05}
Brosnan,~P.   
On motivic decompositions arising from the method of Bialynicki-Birula. 
{\it Invent. Math.} {\bf 161} (2005), no.1, 91--111.

\bibitem[CPSZ]{CPSZ}
Calmès,~B.; Petrov,~V.; Semenov,~N.; Zainoulline,~K.
Chow motives of twisted flag varieties.
{\it Compos. Math.} {\bf 142} (2006), no. 4, 1063–1080.

\bibitem[CM06]{CM06}
Chernousov,~V.; Merkurjev,~A.
Motivic decomposition of projective homogeneous varieties and the Krull-Schmidt theorem.
{\it Transform. Groups} {\bf 11} (2006), no. 3, 371--386.

\bibitem[EG97]{EG97}
Edidin,~D.; Graham,~W.   
Characteristic classes in the Chow ring. 
{\it J. Algebraic Geom.} {\bf 6} (1997), no.3, 431--443.


\bibitem[Fu98]{Fu98}
Fulton,~W.
Intersection theory.
{\it Results in Mathematics and Related Areas. 3rd Series.} A Series of Modern Surveys in Mathematics.
Springer-Verlag, Berlin, 1998, xiv+470 pp.

\bibitem[Gi15]{Gi15}
Gille,~S.
Permutation modules and Chow motives of geometrically rational surfaces.
{\it J. Algebra} {\bf 440} (2015), 443--463.

\bibitem[GS06]{GS06}
Gille,~P.; Szamuely,~T. 
Central Simple Algebras and Galois Cohomology. 
{\it Cambridge Studies in Advanced Mathematics} {\bf 101},
Cambridge Univ. Press, 2006, 430pp.

\bibitem[Ka12]{Ka12}
Karpenko,~N.
Unitary grassmannians.
{\it J. Pure Appl. Algebra} {\bf 216} (2012), no. 12, 2586--2600.

\bibitem[Ka00]{Ka00}
Karpenko,~N.
Criteria of motivic equivalence for quadratic forms and central simple algebras. 
{\it Math. Ann.} {\bf 317} (2000), no.3, 585--611.

\bibitem[Ka96]{Ka96}
Karpenko,~N.   
The Grothendieck-Chow motifs of Severi-Brauer varieties. 
{\it St. Petersburg Math.~J.} {\bf 7} (1996), no.4, {649--661}.

\bibitem[KMRT]{KMRT}
Knus,~M.-A.; Merkurjev,~A.; Rost,~M.; Tignol,~J.-P.  
The Book of Involutions. 
{\it Colloquium Publications} {\bf 44}, 
AMS, Providence, RI, 1998.

\bibitem[Ko91]{Ko91}
K\"ock,~B.   
Chow motif and higher Chow theory of $G/P$. 
{\it Manuscripta Math.} {\bf 70} (1991), 363--372.

\bibitem[LM07]{LM07}
Levine,~M.; Morel,~F.
Algebraic cobordism.
{\it Springer Monogr. Math.}
Springer, Berlin, 2007, xii+244 pp.

\bibitem[Ma68]{Ma68}
Manin,~Y.   
Correspondences, Motives and Monoidal Transformations. 
{\it Math. USSR Sb.} {\bf 6} (1968), 439--470.

\bibitem[MPW]{MPW1}
Merkurjev,~A.; Panin,~I.; Wadsworth,~A.   
Index reduction formulas for twisted flag varieties. I. 
{\it K-Theory J.} {\bf 10} (1996), 517--596.

\bibitem[Se08]{Se08}
Semenov, N. 
Motivic decomposition of a compactification of a Merkurjev-Suslin variety. 
{\it J. Reine Angew. Math.} {\bf 617} (2008), 153--167.

\bibitem[SZ10]{SZ10}
Semenov,~N.; Zainoulline,~K.
Essential dimension of Hermitian spaces.
{\it Math. Ann.} {\bf 346} (2010), no. 2, 499--503.

\bibitem[Sp09]{Sp09}
Springer, T. A.
Linear algebraic groups.
Mod. Birkhäuser Class.
Birkhäuser Boston, Inc., Boston, MA, 2009, xvi+334 pp

\bibitem[To97]{To97}
Totaro,~B.   
The Chow ring of a classifying space. Algebraic $K$-theory 
(Seattle, WA, 1997), 249--281, {\it Proc. Sympos. Pure Math.} {\bf 67}, Amer. Math. Soc., Providence, RI, 1999.

\end{thebibliography}
\end{document}